\documentclass[12pt,a4paper,psamsfonts]{amsart}
\usepackage{fancyhdr}
\usepackage{appendix}
\usepackage{amssymb,amscd,amsxtra,calc}
\usepackage{mathrsfs}
\usepackage{cmmib57}
\usepackage{multirow}
\usepackage[all]{xy}
\usepackage{longtable}
\usepackage[colorlinks,linkcolor=blue,anchorcolor=yellow,citecolor=blue]{hyperref}
\usepackage{tikz}

\theoremstyle{plain}
    \newtheorem{thm}{Theorem}[section]

     \newtheorem{conjecture}[thm]{Conjecture}

    \newtheorem{lemma}[thm]{Lemma}
    \newtheorem{proposition}[thm]{Proposition}
    \newtheorem{question}[thm]{Question}
    \newtheorem{theorem}[thm]{Theorem}

\theoremstyle{definition}
    \newtheorem{definition}[thm]{Definition}

    \newtheorem*{notation*}{Notation and Terminology}
      
    \newtheorem{remark}[thm]{Remark}
    
\theoremstyle{remark}

\newcommand{\bbC}{\mathbb{C}}

\newcommand{\bbP}{\mathbb{P}}

\newcommand{\bbR}{\mathbb{R}}

\newcommand{\bQ}{\mathbb{Q}}
\newcommand{\bR}{\mathbb{R}}
\newcommand{\bZ}{\mathbb{Z}}

\newcommand{\NE}{\overline{\operatorname{NE}}}

\newcommand{\Pic}{\operatorname{Pic}}

\newcommand{\mstriangle}[1]{
\begin{tikzpicture}[x=0.3cm,y=0.3cm]
\draw (-0.4,-0.433) -- (1.4,-0.433);
\draw (-0.2,-0.7794) -- (0.7,0.7794);
\draw (1.2,-0.7794) -- (0.3,0.7794);
\end{tikzpicture}
}
\newcommand{\mssharp}[1]{
\begin{tikzpicture}[x=0.3cm,y=0.3cm]
\draw (-0.8,-0.5) -- (0.8,-0.5);
\draw (-0.8,0.5) -- (0.8,0.5);
\draw (-0.5,-0.8) -- (-0.5,0.8);
\draw (0.5,-0.8) -- (0.5,0.8);
\end{tikzpicture}
}

\makeatletter

\newcommand{\Rmnum}[1]{\expandafter\@slowromancap\romannumeral #1@}
\makeatother

\begin{document}

\title[Bounded cohomology property]
{Bounded cohomology property on a smooth projective surface with Picard number two}
\author{Sichen Li}
\address{
School of Mathematics, East China University of Science and Technology, Shanghai 200237, P. R. China}
\email{\href{mailto:sichenli@ecust.edu.cn}{sichenli@ecust.edu.cn}}
\dedicatory{Dedicated to Professor Sheng-Li Tan on the occasion of his 60th birthday}
\thanks{The research is supported by Shanghai Sailing Program (No. 23YF1409300).}
\begin{abstract}
We say  a smooth projective surface $X$ satisfies the bounded cohomology property  if there exists a positive constant $c_X$ such that $h^1(\mathcal O_X(C))\le c_Xh^0(\mathcal O_X(C))$ for every prime divisor $C$ on $X$.
Let the closed Mori cone $\NE(X)=\bR_{\ge0}[C_1]+\bR_{\ge0}[C_2]$ such that $C_1$ and $C_2$ with $C_2^2<0$ are some curves on $X$.
If either (i) the Kodaira dimension $\kappa(X)\le1$ or (ii) $\kappa(X)=2$, the irregularity $q(X)=0$ and the Iitaka dimension $\kappa(X,C_1)=1$, then we prove that $X$ satisfies the bounded cohomology property.
\end{abstract}

\maketitle

\section{Introduction}
In this note we work over the field $\bbC$ of complex numbers.
By a (negative) curve  on a surface we will mean a reduced, irreducible curve (with negative self-intersection).
By a {\it (-k)-curve}, we mean a negative curve $C$ with $C^2=-k<0$.
A prime divisor $C$ on a surface $X$ is either a nef curve or a negative curve in which case that $h^0(X,\mathcal O_X(C))=1$.

The bounded negativity conjecture (BNC for short)  is one of the most intriguing problems in the theory of projective surfaces and can be formulated as follows.
\begin{conjecture}	
\cite[Conjecture 1.1]{B.etc.13}\label{BNC}
For a smooth projective surface $X$ there exists an integer $b(X)\ge0$ such that $C^2\ge-b(X)$ for every curve $C\subseteq X$.
\end{conjecture}
\begin{definition}
We say that a smooth projective surface $X$ has
\begin{equation*}
b(X)>0	
\end{equation*}
if there is at least one negative curve on $X$.	
\end{definition}
Recall that  a smooth projective surface $X$ satisfies \text{\it bounded cohomology property} if there is a positive constant $c_X$ such that $h^1(\mathcal O_X(C))\le c_Xh^0(\mathcal O_X(C))$ for every curve $C$ on $X$ (cf. \cite[Conjecture 2.5.3]{B.etc.12}).
Notice that $X$ satisfies the BNC if $X$ satisfies the bounded cohomology property (cf. \cite[Proposition 14]{C.etc.17}).
On the one side, there exists  a surface of general type with a larger Picard number which does not satisfy the bounded cohomology property (cf. \cite[Corollary 3.1.2]{B.etc.12}).
On the other side, recent results in \cite{Li19,Li21} have helped to clarify which surfaces satisfy the bounded cohomology property.
In \cite{Li19}, we gave a classification of the smooth projective surface with $\rho(X)=2$ and two negative curves. 
This result motivated us to showed  in \cite{Li21} that every smooth projective surface with $\rho(X)=2$ and $b(X)>0$ satisfies the  bounded cohomology property if  either the Kodaira dimension $\kappa(X)=1$ or ~$X$~has two negative curves.
One may observe that the closed Mori cone $\NE(X)$ of these surfaces is generated by two curves.
This motivates the following question.
\begin{question}
(cf. \cite[Question 3.6]{Li21})
\label{Mainque}
Let $X$ be a smooth projective surface with the Picard number $\rho(X)=r$ and $C_i$ finitely many curves on $X$.
Suppose that $\NE(X)=\sum_{i=1}^r \bR_{\ge0}[C_i]$.
Does $X$ satisfy the bounded cohomology property?	
\end{question}
\begin{remark}
As a start to addressing Question \ref{Mainque},  we show in Proposition \ref{NE} that $X$ has finitely many negatively curves (hence satisfies the BNC) and there is a positive constant $m(X)$ such that $l_C\le m(X)$ for every curve $C$ with $C^2\ne0$ on $X$. 
Here, $l_C$ is defined in Definition \ref{defn1}.
\end{remark} 
Now we give our main result as follows.
\begin{theorem}\label{MainThm}
Let $X$ be a smooth projective surface with the Picard number $\rho(X)=2$ and $b(X)>0$.
Then $X$ satisfies the bounded cohomology property if either (i) $\kappa(X)\le1$ or (ii) $\kappa(X)=2$, the irregularity $q(X)=0$ and the Iitaka dimension $\kappa(X,C_1)=1$.
\end{theorem}
A Mori dream surface is a normal $\bQ$-factorial surface $X$ with finitely generated divisor class group $\mathrm{Cl}(X)$ whose nef cone is generated by finitely many semiample classes, and $h^1(X, \mathcal O_X)=0$ (cf. \cite{HK00}).
Ro\'e asked us the following question.
\begin{question}
Does every Mori dream surface satisfy the bounded cohomology property?
\end{question}
\section{Study for $l_C$}
We first recall the definition of $l_C$ for curves $C$ on $X$ in \cite[Defintion 2.2]{Li21}.
\begin{definition}
\label{defn1}
Let $X$ be a smooth projective surface.
\begin{enumerate}
\item[(1)] For every $\bR$-divisor $D$ with $D^2\ne0$ on $X$, we define a value $l_D$ of $D$ as follows:
\begin{equation*}
                                                       l_D:=\frac{(K_X\cdot D)}{\max\bigg\{ 1, D^2\bigg\}}.
\end{equation*}
\item[(2)] For every $\bR$-divisor $D$ with $D^2=0$ on $X$, we define a value $l_D$ of $D$ as follows:
\begin{equation*}
                                   l_D:=\frac{(K_X\cdot D)}{\max\bigg\{1,h^0(\mathcal O_X(D))\bigg\}}	
\end{equation*}
\end{enumerate}
\end{definition}
\begin{remark}
 Ciliberto et al asked in \cite[Question 4]{C.etc.17} that whether there exists a positive constant $m(X)$ such that $l_D\le m(X)$ for every curve $D$ with $D^2>0$ on $X$.
When the closed Mori cone $\NE(X)$ is generated by finitely many curves, this question has an affirmative answer (cf. Proposition \ref{NE}).
Moreover, the question motivated us to give a numerical characterization of the  bounded cohomology property in \cite{Li21}.
\end{remark}
\begin{proposition}
\label{keyprop}
\cite[Proposition 2.3]{Li21}
Let $X$ be a smooth projective surface.
If there exists a positive constant $m(X)$ such that $l_C\le m(X)$ for every curve $C$ on $X$ and $|D^2|\le m(X)h^0(\mathcal O_X(D))$ for every curve $D$ with either $D^2<0$ or $l_D>1$ and $D^2>0$ on $X$, then $X$ satisfies the  bounded cohomology property.
\end{proposition}
\begin{remark}
We say a surface $X$ satisfies \emph{uniformly boundedness for $l_C$ and $C^2$} respectively if there exists a positive constant $m(X)$ such that $l_C\le m(X)$ and $|C^2|\le m(X)h^0(\mathcal O_X(C))$ respectively for every curve $C$ on $X$.
Proposition \ref{keyprop} established that the uniformly boundedness for $l_C$ and $C^2$ imply the  bounded cohomology property.
\end{remark}
The following is due to Serre duality.
\begin{proposition}\label{iq1}
Let $C$ be a curve on a smooth projective surface $X$.
Then
\begin{equation*} 
 h^2(\mathcal O_X(C))-\chi(\mathcal O_X)\le q(X)-1.	
\end{equation*}
Here, $q(X)$ is the irregularity of $X$.
\end{proposition}
Thanks to Proposition \ref{keyprop}, we can quickly prove the following result.
\begin{proposition}\label{rho(X)=1}
 Every smooth projective surface $X$ with the Picard number $\rho(X)=1$ satisfies the bounded cohomology property.
\end{proposition}
\begin{proof}
We may assume that $\kappa(X)\le0$ or $\kappa(X)=2$.
Note that every curve $C$ on $X$ is nef and big.
If $\kappa(X)\le0$, then $h^1(\mathcal O_X(C))=0$ follows from Kawamata-Viehweg vanishing theorem (cf. \cite[Theorem 2.64]{KM98}).
If $\kappa(X)=2$, then we may assume that $K_X=aH$ and $C=dH$ with $a,d\in\bR_{>0}$.
Now $(K_X-C)C=(a-d)dH^2$.
If $d\ge a$, then $(K_X-C)C\le0$.
Then by Riemann-Roch theorem and Proposition \ref{iq1} imply that
\begin{equation*}\begin{split}
            h^1(\mathcal O_X(C))&=h^0(\mathcal O_X(C))+h^2(\mathcal O_X(C))+\frac{(K_X-C)C}{2}-\chi(\mathcal O_X)\\ & <(q(X)+1)h^0(\mathcal O_X(C)).
\end{split}\end{equation*}
If $d<a$, then $C^2\le a^2H^2h^0(\mathcal O_X(C))$.
This ends the proof by Proposition \ref{keyprop}.
\end{proof}
\begin{proposition}
\label{C^2=0}
Let $C$ be a curve on a smooth projective surface $X$.
Suppose $C^2=0$.
Then there exists a constant $a_C$ such that $l_D\le a_C$ for every $n>0$ and  every curve $D\in |nC|$ (note that being a curve implies $D$ is reduced and irreducible).
\end{proposition}
\begin{proof}
Take a curve $D\in |nC|$.
If $\kappa(X,C)=0$, then $D=C$.
Now we may assume that $\kappa(X,C)=1$.
By \cite[Corollary 2.1.38]{Lazarsfeld04}, there exists a constant $a>0$ and sufficiently large integer $m=m(C)>0$ such that
\begin{equation*}
                     an\le h^0(X,nC)
\end{equation*}
for all sufficiently large $n\ge m$.
So  we have
\begin{equation*}\begin{split}
	 l_D&=\frac{n(K_X\cdot C)}{h^0(X,nC)}
	       \\&\le \max\bigg\{\max_{1\le k\le  m}\big\{l_{kC}\big\}, \frac{(K_X\cdot C)}{a}\bigg\}:=a_C.
\end{split}\end{equation*}
This completes the proof of Proposition \ref{C^2=0}.
\end{proof}
The following result answers \cite[Question 4]{C.etc.17}.
\begin{proposition}\label{NE}
Let $X$ be a smooth projective surface and finitely many curves $C_i$ on $X$ with the Picard number $\rho(X)=r$.
Suppose that $\NE(X)=\sum_{i=1}^r\bR_{\ge0}[C_i]$.
Then $X$ has finitely many negative curves and there exists a positive constant $m(X)$ such that $l_D\le m(X)$ for every curve $D$ with $D^2\ne0$ on $X$.	
\end{proposition}
\begin{proof}
Let $D=\sum_{i=0}^ra_iC_i$ be a curve with $D^2\ne0$ on $X$ and each $a_i\ge0$.
If $D^2<0$, then $0>D^2=D\cdot (\sum_i a_iC_i)$ so $D=C_i$ for some $i$, and thereby conclude that there are only finitely many negative curves, and also $l_D\le \max_{1\le i\le \rho(X)}\{l_{C_i}\}$.
Now we may assume that $D^2>0$.
As a result, $D\cdot C_i\ge1$.
Then we have
\begin{equation*}\begin{split}
 l_D&=\frac{K_X\cdot D}{D^2}	
 \\&=\frac{\sum_{i=1}^ra_i(K_X\cdot C_i)}{\sum_{i=1}^ka_i(D\cdot C_i)}
 \\&\le \max_{1\le i\le r}\bigg\{\frac{|K_X\cdot C_i|}{D\cdot C_i}\bigg\}\\&\le \max_{1\le i\le r}\big\{|K_X\cdot C_i|\big\}:=C(X).
\end{split}\end{equation*}
In all, there exists a positive constant $m(X)=\max\big\{C(X), \max_{j\in J}\{l_{C_j}\}\big\}$ such that $l_D\le m(X)$ for every curve $D$ with $D^2\ne0$ on $X$.
\end{proof}
\section{The Proof of Theorem \ref{MainThm}}
\begin{lemma}\label{ruled}
Let $X$ be a geometrically ruled surface over a smooth curve $B$ of genus $g$, with invariant $e$.
Let $C\subset X$ be the unique section, and let $f$ be a fibre.
If $b(X)>0$, then $X$ satisfies the bounded cohomology property.
\end{lemma}
\begin{proof}
We first show that $e>0$ if and only if $b(X)>0$.
By \cite[Proposition V.2.3 and 2.9]{Hartshorne77},
	\begin{equation*}
	\Pic 	X\cong \bZ C\oplus\pi^*\Pic B, C\cdot f=1, f^2=0, C^2=-e, K_X\equiv -2C+(2g-2-e)f.
	\end{equation*}
If $e>0$, then $C^2<0$ (hence $b(X)>0$).
Conversely, we may assume that $b(X)>0$.
Then there exists a negative curve $C'=aC+a'f$ on $X$.
As a result, $C'=C$ by \cite[Propositions V.2.20 and 2.21]{Hartshorne77}.

Now by \cite[Proposition V.2.20(a)]{Hartshorne77},  $a>0, a'\ge ae$.
As a result, every curve $D=aC+a'f(\ne C,f)$ has $D^2>0$, $(C\cdot D)\ge0$ and $(f\cdot D)>0$. Thus,
\begin{equation*}\begin{split}
                                         l_D&=\frac{(K_X\cdot D)}{D^2}\\&\le\frac{2(C\cdot D)+|2g-2-e|(f\cdot D)}{a(C\cdot D)+a'(f\cdot D)}\\&\le\mathrm{max}\bigg\{\frac{2}{a},\frac{|2g-2-e|}{a'}\bigg\}.
\end{split}\end{equation*}
Here, $a$ and $b$ are positive integers.
So there exists a positive constant $m=\max\{2,|2g-2-e|, l_C, l_f\}$ such that $l_F\le m$ for every curve $F$ on $X$.

If $a'\ge 2g-2-e$, then
\begin{equation}\label{eq3.2}
                        (K_X-D)D=-(2+a)(C\cdot D)+(2g-2-e-a')(f\cdot D)\le0.
\end{equation}
By Riemann-Roch theorem, (\ref{eq3.2}) and Proposition \ref{iq1} imply that
\begin{equation*}\begin{split}
            h^1(\mathcal O_X(D))&=h^0(\mathcal O_X(D))+h^2(\mathcal O_X(D))+\frac{(K_X-D)D}{2}-\chi(\mathcal O_X)\\&< (q(X)+1)h^0(\mathcal O_X(D)).
\end{split}\end{equation*}
If $a'<2g-2-e$, then $a<(2g-2-e)e^{-1}$ by $a'\ge ae$. As a result, $D^2<2(2g-2-e)^2e^{-1}$. Hence, by  Proposition \ref{keyprop}, $X$ satisfies  the bounded cohomology property.
\end{proof}
\begin{lemma}\label{kappa<0}
Let $X$ be a smooth projective surface with $\rho(X)=2$ and $b(X)>0$.
Then $X$ satisfies the bounded cohomology property if $\kappa(X)=-\infty$.
\end{lemma}
\begin{proof}
By the Enriques-Kodaira classification of relatively minimal surfaces (cf. \cite{Hartshorne77,KM98}) and $\kappa(X)=-\infty$,  $X$ is either a ruled surface (this case then follows from Lemma \ref{ruled}) or one point  blow up  of $\mathbb P^2$.
Now we assume that $X$ is one-point blow up of $\bbP^2$  with a exceptional curve $E$ and $\mathrm{Pic}(\mathbb P^2)=\bZ[H]$, where $H=\mathcal O_{\mathbb P^2}(1)$.
 Then  $K_X=\pi^{*}(-3H)+E$ and $C=\pi^*(dH)-mE$, where $m:=\mathrm{mult}_p(\pi_*C)$ and $C$ is a curve on $X$.
 Note that $d\ge m$ since $\pi_*C$ is a plane projective curve.
 Thus, every curve $C$ (not $E$) on $X$ has $C^2\ge0$ and then $C$ is nef.
 Since $-K_X$ is ample, $C-K_X$ is ample.
 Therefore, by Kodaira vanishing theorem, $h^1(\mathcal O_X(C))=0$. 
So $X$ satisfies the bounded cohomology property.
 \end{proof}
Now we begin to answer Question \ref{Mainque} for the case $\kappa(X)\ge0$.
\begin{proposition} \label{>0}
\cite[Proposition 3.1]{Li21}
Let $X$ be a smooth projective surface with $\rho(X)=2$. Then the following statements hold.
\begin{enumerate}
\item[(i)] $\NE(X)=\bbR_{\ge0}[f_1]+\bbR_{\ge0}[f_2]$, $f_1^2\le0, f_2^2\le0$ and $f_1\cdot f_2>0$. Here, $f_1, f_2$ are extremal rays.
\item[(ii)] If a curve $C$ has $C^2\le0$, then $C\equiv af_1$ or $C\equiv a'f_2$ for some $a,a'\in\mathbb R_{>0}$.
\item[(iii)] Suppose a divisor $D\equiv a_1f_1+a_2f_2$ with $a_1,a_2>0$ in (i).  Then $D$ is big. Moreover, if $D$ is a curve, then $D$ is nef and big and $D^2>0$.
\end{enumerate}
\end{proposition}
\begin{proposition}\label{l_D}
If  $\NE(X)=\bR_{\ge0}[C_1]+\bR_{\ge0}[C_2]$ with $C_1,C_2$ are  some curves on $X$, then $X$ satisfies the uniformly boundedness for $l_C$.
\end{proposition}
\begin{proof}
If $D^2\ne0$, then the proof follows from Proposition \ref{NE}.
If $D^2=0$, then $D=aC_i$ with $C_i^2=0$ by Proposition \ref{>0}.
We may assume that $a$ is a positive integer since $D$ is a curve.
So the proof follows from Proposition \ref{C^2=0}.
\end{proof}
\begin{proposition}\label{reduced}
Let $X$ be a smooth projective surface with $\rho(X)=2$.
To answer Question \ref{Mainque}, we may assume that the canonical divisor $K_X=aC_1+a'C_2$ with $a,a'\in\bR_{\ge0}$ and every curve $D=a_1C_1+a_2C_2$ with $a_1>a$ and $0<2a_2<a'$.	
\end{proposition} 
\begin{proof}
By Propositions \ref{keyprop} and \ref{l_D}, we only need to show that there exists a positive constant $m(X)$ such that $D^2\le m(X)h^0(\mathcal O_X(D))$ for every curve $D$ with $l_D>1$ and $D^2>0$ on $X$.
Now let $\NE(X)=\bR_{\ge0}[C_1]+\bR_{\ge0}[C_2]$ with $C_1$ and $C_2$ are some curves on $X$.
If $C_1^2<0$ and $C_2^2<0$, the proof follows from \cite[Lemma 3.4]{Li21}.

Now we may assume that $C_1^2=0$ and $C_2^2<0$.	
If either $\kappa(X)=-\infty$ or $\kappa(X)=1$, the proof follows from either   Lemma \ref{kappa<0}  or \cite[Lemma 3.2]{Li21} respectively.
Now we may assume that $\kappa(X)=0$ or $\kappa(X)=2$.
In fact, we only assume that $K_X=aC_1+a'C_2$ with $a,a'\ge0$ and $D=a_1C_1+a_2C_2$ with $a_1,a_2>0$.

If $a_1>a$ and $a_2>a'$, then 
\begin{equation*}
(K_X-D)D=(a-a_1)(D\cdot C_1)+(a'-a_2)(D\cdot C_2)<0.	
\end{equation*}
This and Proposition \ref{iq1} imply that 
\begin{equation*}\begin{split}
                                             h^1(\mathcal O_X(D))&=h^0(\mathcal O_X(D))+h^2(\mathcal O_X(D))+\frac{(K_X\cdot D)-D^2}{2}-\chi(\mathcal O_X)\\&<(q(X)+1)h^0(\mathcal O_X(D)).
\end{split}\end{equation*}
Note that $D^2>0$ implies that
\begin{equation*}
D^2=2a_1a_2(C_1\cdot C_2)-a_2^2(-C_2^2)>0.	
\end{equation*}
So 
\begin{equation}\label{eq2}
\frac{a_1}{a_2}>\frac{-C_2^2}{2(C_1\cdot C_2)}.
\end{equation}

If $a_1\le a$, then $a_2\le 2a(C_1\cdot C_2)(-C_2^2)^{-1}$ by (\ref{eq2}).
As a result,
\begin{equation*}
 D^2\le 4a^2(C_1\cdot C_2)^2(-C_2^2)^{-1}.	
\end{equation*}
Now we may assume that $a_1>a$ and $a_2<a'$.
If $2a_2\ge a'$, then
\begin{equation*}\begin{split}
(K_X-D)D&=((a-a_1)C_1+(a'-a_2)C_2)(a_1C_1+a_2C_2)
\\&=(aa_2+a_1(a'-2a_2))(C_1\cdot C_2)+a_2(a_2-a')(-C_2^2)
\\&<aa'(C_1\cdot C_2).
\end{split}\end{equation*}
This and Proposition \ref{iq1} imply that 
\begin{equation*}\begin{split}
                                             h^1(\mathcal O_X(D))&=h^0(\mathcal O_X(D))+h^2(\mathcal O_X(D))+\frac{(K_X\cdot D)-D^2}{2}-\chi(\mathcal O_X)\\&\le (q(X)+ab(C_1\cdot C_2))h^0(\mathcal O_X(D)).
\end{split}\end{equation*}

Therefore, to answer  Question \ref{Mainque}, we may assume that $K_X=aC_1+a'C_2$ with $a,a'\in\bR_{\ge0}$ and $D=a_1C_1+a_2C_2$ with $a_1>a$ and $0<2a_2<a'$.
\end{proof}
\begin{lemma}\label{kappa=0}
Let $X$ be a smooth projective surface with $\rho(X)=2$ and $b(X)>0$.
If $\kappa(X)=0$, then $X$ satisfies the bounded cohomology property.
\end{lemma}
\begin{proof}
If $K_X$ is nef, then $K_X\equiv 0$.
As a result, $l_C=0$ for every curve $C$ on $X$.
By the adjunction formula, $C^2\ge-2$.
By Riemann-Roch Theorem and Proposition \ref{iq1}, we have
\begin{equation*}\begin{split}
 2h^1(\mathcal O_X(C))&=2h^0(\mathcal O_X(C))+2h^2(\mathcal O_X(C))-2\chi(\mathcal O_X)-C^2
 \\&\le 2(q(X)+1)h^0(\mathcal O_X(C)).
\end{split}\end{equation*}
Therefore, $X$ satisfies the bounded cohomology property.

Now we assume that $K_X$ is not nef. 
Then by the minimal model program in dimension two (see eg. \cite[Flowchart 1-2-18]{Matsuki01}), $\pi: X\to Y$ is one-point blow-up of a K3 surface or an abelian surface $Y$ with $\rho(Y)=1$ and an exceptional curve $E$.
Then $K_X\equiv \pi^*K_Y+E$.
As a result, $K_X\equiv C_2$  since $\NE(X)=\bR_{\ge0}[C_1]+\bR_{\ge0}[C_2]$ and $C_2^2<0$.
So $a=0$ and $a'=1$.
By  Proposition \ref{reduced}, we may assume every curve $D$ has
\begin{equation}\label{eq}
D=a_1C_1+a_2C_2, a_1>0, 0<2a_2<1.	
\end{equation}
Note that $D$ is big by Proposition \ref{>0}.
This implies that $D\cdot C_1,D\cdot C_2\in\bZ_{>0}$ as $D$ is a curve.
Therefore, we may  assume that $a_1,a_2\in\bZ_{>0}$ after replacing $D$ by $cD$ with some $c\in\bZ_{>0}$. 
In fact, $c=(C_1\cdot C_2)^2+C_1^2>0$.
So $2a_2\ge2$ which leads a contradiction as $2a_2<1$ by Proposition \ref{reduced}.
\end{proof}
\begin{lemma}\label{kappa=2}
Let $X$ be a smooth projective surface with $\rho(X)=2$.
If  $\kappa(X)=2, q(X)=0$ and $\kappa(X,C_1)=1$, then  $X$ satisfies the bounded cohomology property.
\end{lemma}
\begin{proof}
By Proposition \ref{reduced}, we may assume that $K_X=aC_1+a'C_2$ and every curve $D=a_1C_1+a_2C_2$ with $a_1>a$ and $0<2a_2<a'$.
On the one hand, since $q(X)=0$, we have
\begin{equation}\label{eq4}
h^0(\mathcal O_X(D))=h^0(\mathcal O_X(a_1C_1+a_2C_2))
	                 \ge h^0(\mathcal O_X(a_1C_1)).
\end{equation}
By \cite[Corollary 2.1.38]{Lazarsfeld04}, there exists a positive constant $c=(c_X^{-1}a')\cdot(C_1\cdot C_2)$ such that 
\begin{equation}\label{eq5}
	h^0(X,\mathcal O_X(a_1C_1))\ge (c_X^{-1}a')\cdot (C_1\cdot C_2)a_1.
\end{equation}
By (\ref{eq4}) and (\ref{eq5}), we have
\begin{equation}\label{eq6}
 h^0(\mathcal O_X(D))\ge  (c_X^{-1}a')\cdot(C_1\cdot C_2)a_1.
\end{equation}
On the other hand,
\begin{equation}\label{eq7}
D^2=2a_1a_2(C_1\cdot C_2)-a_2^2(-C_2^2)	
\le a'(C_1\cdot C_2)a_1,
\end{equation}
where using $0<2a_2<a'$ and $-C_2^2>0$.
Therefore, by (\ref{eq6}) and (\ref{eq7}), we have
\begin{equation*}
   D^2\le c_Xh^0(\mathcal O_X(D)).
\end{equation*}
Hence, $X$ satisfies the bounded cohomology property by Propositions \ref{keyprop} and \ref{l_D}.
 \end{proof}
\begin{proof}[Proof of  Theorem \ref{MainThm}]
It follows from Lemmas \ref{kappa<0}, \ref{kappa=0}, \ref{kappa=2} and \cite[Lemma 3.2]{Li21}.
\end{proof}
\section*{Acknowledgments}
The author would like to thank  Meng Chen and Sheng-Li Tan for their constant encouragement, Joaquim Ro\'e for his comments and the anonymous referee for several suggestions.


\begin{thebibliography}{99}


\bibitem{B.etc.12}
T. Bauer, C. Bocci, S. Cooper, S. D. Rocci, M. Dumnicki, B. Harbourne, K. Jabbusch, A. L. Knutsen,
A. K$\mathfrak{\ddot{u}}$ronya, R. Miranda, J. Ro$\mathrm{\acute{e}}$, H. Schenck, T. Szemberg, and Z. Teithler,
\emph{Recent developments and open problems in linear series,}
In: Contributions to Algebraic Geometry, p. 93-140, EMS Ser. Congr. Rep., Eur. Math. Soc., Z\"{u}rich, 2012.

\bibitem{B.etc.13}
T. Bauer, B. Harbourne, A. L. Knutsen, A. K$\mathrm{\ddot{u}}$ronya, S. M$\mathrm{\ddot{u}}$ller-Stach, X. Roulleau, and T. Szemberg,
\emph{Negative curves on algebraic surfaces,}
Duke Math. J. \textbf{162}(10)(2013), 1877-1894.

\bibitem{C.etc.17}
C. Ciliberto, A. L. Knutsen, J. Lesieutre, V. Lozovanu, R. Miranda, Y. Mustopa, and D. Testa,
\emph{A few questions about curves on surfaces,} Rend. Circ. Mat. Palermo. II. Ser. \textbf{66} (2)(2017),195-204.

\bibitem{Hartshorne77}
R. Hartshorne,
\emph{Algebraic Geometry,} GTM \textbf{52}, Springer-Verlag, New York, 1977.

\bibitem{HK00}
Y. Hu and S. Keel,
\emph{Mori dream spaces and GIT},
Michigan Math. J. \textbf{48}(2000), 331-348.

\bibitem{KM98}
J. Koll$\mathrm{\acute{a}}$r and S. Mori,
\emph{Birational geometry of algebraic varieties,}
 Cambridge Tracts in Mathematics, \textbf{134}, Cambridge University Press, Cambridge, 1998.

\bibitem{Lazarsfeld04}
R. Lazarsfeld, \emph{Positivity in algebraic geometry, I}, Ergebnisse der
  Mathematik und ihrer Grenzgebiete. \textbf{48}, Springer-Verlag, Berlin, 2004.

\bibitem{Li19}
S. Li,
\emph{A note on a smooth projective surface with Picard number 2,}
Math. Nachr. \textbf{292} (2019), no.~12, 2637-2642.

\bibitem{Li21}
S. Li,
\emph{Bounding cohomology on a smooth projective surface with Picard number 2,}
Commun. Algebra \textbf{49} (2021), no. 7, 3140-3144.

\bibitem{Matsuki01}
K. Matsuki,
\emph{Introduction to the Mori program},
Universitext, Springer Verlag, Berlin. 2001.



\end{thebibliography}
\end{document}